\documentclass{amsart}
\usepackage{amssymb,xspace,enumitem}

\usepackage[usenames,dvipsnames]{color}

\usepackage[colorlinks,citecolor=blue,urlcolor=blue,linkcolor=blue]{hyperref}

\usepackage[T1]{fontenc} 
\usepackage[utf8]{inputenc} 

\usepackage{soul}

\newtheorem{theorem}{Theorem}[section]

\newtheorem{thm}[theorem]{Theorem}
\newtheorem{fact}[theorem]{Fact}

\newtheorem{prop}[theorem]{Proposition}
\newtheorem{claim}[theorem]{Claim}

\newtheorem{lemma}[theorem]{Lemma}

\newtheorem{cor}[theorem]{Corollary}

\theoremstyle{definition}
\newtheorem{definition}[theorem]{Definition}

\newtheorem{remark}[theorem]{Remark}

\theoremstyle{remark}
\newtheorem*{aside}{Aside}

\newcommand{\DII}{\Delta^0_2}
\newcommand{\NN}{\omega}

\newcommand{\sub}{\subseteq}
\newcommand{\sN}[1]{_{#1\in \NN}}

\DeclareMathOperator{\uhr}{\upharpoonright}

\newcommand{\SI}[1]{\Sigma^0_{#1}}
\newcommand{\PI}[1]{\Pi^0_{#1}}
\newcommand{\PPI}{\PI{1}}

\newcommand{\bi}{\begin{itemize}}
\newcommand{\ei}{\end{itemize}}
\newcommand{\bc}{\begin{center}}
\newcommand{\ec}{\end{center}}

\newcommand{\Halt}{{\ES'}}
\newcommand{\ES}{\emptyset}

\newcommand{\tp}[1]{2^{#1}}
\newcommand{\ex}{\exists}
\newcommand{\fa}{\forall}

\newcommand{\la}{\langle}
\newcommand{\ra}{\rangle}

\newcommand{\leT}{\le_{\mathrm{T}}}

\newcommand{\n}{\noindent}

\newcommand{\sss}{\sigma}
\newcommand{\aaa}{\alpha}

\renewcommand{\phi}{\varphi}
\renewcommand{\setminus}{\smallsetminus}
\renewcommand{\hat}{\widehat}

\newcommand \seq[1]{{\left\langle{#1}\right\rangle}}

\newcommand\+[1]{\mathcal{#1}}

\newcommand{\wt}{\widetilde}
\newcommand{\ol}{\overline}
\newcommand{\ape}{\hat{\ }}

\newcommand{\lra}{\leftrightarrow}

\newcommand{\RA}{\Rightarrow}
\newcommand{\LA}{\Leftarrow}
\newcommand{\DA}{\downarrow}

\newcommand{\rapf}{\n $\RA:$\ }
\newcommand{\lapf}{\n $\LA:$\ }

\newcommand{\sssl}{\ensuremath{|\sigma|}}

\newcommand{\dom}{\ensuremath{\mathrm{dom}}}

\newcommand{\frb}{\mathfrak{b}}
\newcommand{\frd}{\mathfrak{d}}

\newcommand{\fru}{\mathfrak{u}}
\newcommand{\set}[2]{\{#1 \colon #2\}}

\newcommand{\B}{\mathbb B}

{\begin{proof}}{\end{proof}}

\begin{document}

\title[Maximal towers and ultrafilter bases in computability]
{Maximal towers and ultrafilter bases\\ in computability theory}
\date{\today}

\author[Lempp]{Steffen Lempp}
\author[Miller]{Joseph S.~Miller}
\author[Nies]{\\Andr\'e Nies}
\author[Soskova]{Mariya I.~Soskova}

\address[Lempp, Miller, Soskova]{Department of Mathematics,
University of Wisconsin--Madi\-son,
480 Lincoln Drive, Madison, WI 53706,
USA}
\email{lempp@math.wisc.edu}
\email{jmiller@math.wisc.edu}
\email{msoskova@math.wisc.edu}

\address[Nies]{School of Computer Science,
University of Auckland,
Private bag 92019,
Auckland,
New Zealand}
\email{andre@cs.auckland.ac.nz}

\thanks{Lempp was partially supported by Simons Collaboration Grant for
Mathematicians \#626304. Miller was partially supported by grant \#358043
from the Simons Foundation. Nies was partially supported by the Marsden fund
of New Zealand, grant 19-UOA-346. Soskova was partially supported by NSF
Grant DMS-1762648. Miller and Nies were partially supported by NSF Grant
DMS-2053848. The authors thank J\"org Brendle and Noam Greenberg for helpful
discussions during a workshop at the Casa Matem\'atica Oaxaca in August 2019,
where this research received its initial impetus. The authors would also like
to thank the referees for helpful comments.}


\begin{abstract}
The tower number~$\mathfrak t$ and the ultrafilter number~$\fru$ are
cardinal characteristics from set theory. They are based on combinatorial
properties of classes of subsets of~$\omega$ and the almost inclusion
relation $\sub^*$ between such subsets. We consider analogs of these
cardinal characteristics in computability theory.

We say that a sequence $(G_n)_{n \in \mathbb N}$ of computable sets is a
\emph{tower} if $G_0 = \mathbb N$, $G_{n+1} \subseteq^* G_n$, and
$G_n\setminus G_{n+1}$ is infinite for each~$n$. A tower is \emph{maximal}
if there is no infinite computable set contained in all~$G_n$. A tower
$\seq{G_n}\sN n$ is an \emph{ultrafilter base} if for each computable~$R$,
there is~$n$ such that $G_n \sub^* R$ or $G_n \subseteq^* \ol R$; this
property implies maximality of the tower. A sequence $(G_n)_{n \in \mathbb
N}$ of sets can be encoded as the ``columns'' of a set $G\subseteq \mathbb
N$. Our analogs of~$\mathfrak t$ and $\mathfrak u$ are the mass problems of
sets encoding maximal towers, and of sets encoding towers that are
ultrafilter bases, respectively. The relative position of a cardinal
characteristic broadly corresponds to the relative computational complexity
of the mass problem. We use Medvedev reducibility to formalize relative
computational complexity, and thus to compare such mass problems to known
ones.

We show that the mass problem of ultrafilter bases is equivalent to the
mass problem of computing a function that dominates all computable
functions, and hence, by Martin's characterization, it captures highness.
On the other hand, the mass problem for maximal towers is below the mass
problem of computing a non-low set. We also show that some, but not all,
noncomputable low sets compute maximal towers: Every noncomputable (low)
c.e.\ set computes a maximal tower but no 1-generic $\Delta^0_2$-set does so.

We finally consider the mass problems of maximal almost disjoint, and of
maximal independent families. We show that they are Medvedev equivalent to
maximal towers, and to ultrafilter bases, respectively.
\end{abstract}

\maketitle
\setcounter{tocdepth}{1}
\tableofcontents
\setcounter{tocdepth}{2}

\section{Introduction}

%
%
%
%
%
%

Cardinal characteristics measure how far the set-theoretic universe deviates
from satisfying the continuum hypothesis. They are natural cardinals greater
than~$\aleph_0$ and at most~$2^{\aleph_0}$. For instance, the \emph{bounding
number}~$\frb$ is the least size of a collection of functions $f\colon \NN
\to \NN$ such that no single function dominates the entire
collection.\footnote{This is less commonly, but perhaps more sensibly, called
the \emph{unbounding number}.}
Related is the \emph{dominating number}~$\frd$, the least size of a
collection of functions $f\colon \NN \to \NN$ such that every function is
dominated by some function in the collection.
Here, for functions $f, g \colon \NN \to \NN$, we say that~$g$
\emph{dominates}~$f$ if $g(n) \ge f(n)$ for sufficiently large~$n$. An
important program in set theory is to prove less than or equal-relations
between characteristics in ZFC, and to separate them in suitable forcing
extensions.

Analogs of cardinal characteristics in computability theory were first
studied by Rupprecht~\cite{Rupprecht:thesis,Rupprecht:10} and further
investigated by Brendle, Brooke-Taylor, Ng, and
Nies~\cite{Brendle.Brooke.ea:14}.
An article by Greenberg, Kuyper, and Turetsky~\cite{Greenberg.etal:18}, in
part based on Rupprecht's work, provides a systematic approach to the two
connected settings of set theory and computability, at least for certain
types of cardinal characteristics. The relevant characteristics are given by
binary relations, such as the domination relation~$\le^*$ between functions;
their computability-theoretic analogs are ordered by reducibilities that
measure relative computability. A well-understood example of this is how the
relation~$\le^*$ gives rise to the bounding number~$\frb(\le^*)$ and the
dominating number $\frd(\le^*)$, and their analogs in computability, which
are highness and having hyperimmune degree. A~general reference in set theory
is the survey paper by Blass~\cite{Blass:10}. The brief survey by
Soukup~\cite{Soukup:18} contains a diagram displaying the ZFC inequalities
between the most important characteristics in this setting, along with
$\frb(\le^*)$ and $\frd(\le^*)$.

In this paper, we consider cardinal characteristics that do not fit into the
framework of Rupprecht, and Greenberg, Kuyper and
Turetsky~\cite{Greenberg.etal:18}. In particular, we initiate the study of
the computa\-bility-theoretic analogs of the ultrafilter, tower, and
independence numbers. These characteristics are defined in the setting of
subsets of~$\omega$ up to almost inclusion~$\sub^*$; we give definitions
below.

The \emph{ultrafilter number}~$\fru$ is the least size of a subset {of
$[\omega]^{\omega}$} with upward closure a nonprincipal ultrafilter on
$\omega$. We note that one cannot in general require here that the subset is
linearly ordered by~$\subseteq^*$: Recall that an ultrafilter~$F$ on~$\omega$
is a $P$-point if for each partition $\seq {C_n}$ of~$\omega$ such that $C_n
\not \in F$ for each~$n$, there is $A \in F$ such that $C_n \cap A$ is finite
for each~$n$. An ultrafilter with a linear base is a $P$-point. Shelah (see
Wimmers~\cite{Wimmers:82}) has shown that it is consistent with ZFC that
there are no $P$-points. So it is consistent with ZFC that the version
of~$\fru$ relying on linear bases would be undefined.


The \emph{tower number}~$\mathfrak t$ is the minimum size of a subset of
{$[\omega]^\omega$ that is linearly ordered by~$\subseteq^*$} and cannot be
extended by adding a new element below all given elements. To~define the
\emph{pseudointersection number}~$\mathfrak p$, the requirement in the
definition of towers that the sets in the class be linearly ordered under
$\sub^*$ is weakened to requiring that every finite subset of the class has
an infinite intersection. So, trivially, $\mathfrak p \le \mathfrak t$. In
celebrated work, Malliaris and Shelah~\cite{Malliaris.Shelah:13} showed (in
ZFC) that $\mathfrak p = \mathfrak t$ (see also~\cite{Soukup:18}). It is not
hard to see that ZFC proves $\mathfrak t \le \fru $. It is consistent that
$\mathfrak t < \fru $ (see~\cite{Blass:10} for both statements).

A class~$\+C$ of subsets of~$\NN$ is \emph{independent} if any intersection
of finitely many sets in~$\+C$ or their complements is infinite. The
\emph{independence number}~$\mathfrak i$ is the least cardinal of a maximal
independent family.
There has been much work recently on~$\mathfrak i$ in
set theory, in particular, the descriptive complexity of maximal independent
families, such as in Brendle, Fischer, and
Khomskii~\cite{Brendle.Fischer.etal:19}.


\subsection{Comparing the complexity of the analogs in computability}

The main setting for our analogy is given by the Boolean algebra of
computable sets modulo finite differences. We consider maximal towers, the
closely related maximal almost disjoint sets, and thereafter ultrafilter
bases and maximal independent sets. As already demonstrated in the
above-mentioned papers~\cite{Brendle.Brooke.ea:14, Greenberg.etal:18,
Rupprecht:thesis, Rupprecht:10},
the relative position of a cardinal characteristic tends to correspond to the
relative computational complexity of the associated class of objects.

The usual formal definitions of computation relative to an oracle only
directly apply to functions $f \colon \NN \to \NN$, and hence to subsets
of~$\NN$ (simply called sets from now on), which can be identified with their
characteristic functions. The complexity of other objects is studied
indirectly, via names that are functions on~$\NN$ giving discrete
representations of the object in question. A particular choice of names has
to be made. For instance, real numbers can be named by rapidly converging
Cauchy sequences of rational numbers.

The witnesses for cardinal characteristics are always uncountable. In
contrast, in our setting, the analogous objects are countable. They will be
considered as sequences of sets rather than unordered collections. For, a
single set~$X$ can be used as a name for such a sequence of sets: Let
$X^{[n]}$ denote the ``column'' $ \{ u \colon \langle u, n \rangle \in X
\}$.\footnote{For definiteness, we employ the usual computable Cantor pairing
function $\langle x,n \rangle$. Note that $\langle x,n \rangle \ge x,n $.
This property is useful in simplifying notation in some of the constructions
below.} To every set~$X$, we can associate a sequence $\seq{X_n}\sN n$ in a
canonical way by setting $X_n = X^{[n]}$. (When introducing terminology, we
will sometimes ignore the difference between $\seq{X_n}\sN n$ and~$X$.) An
alternate viewpoint is that a set~$X$ is a name for the unordered collection
of sets in its coded sequence. Although such a name includes more information
than is in the unordered family, this information is suppressed when we
quantify over all names; our results can be read in this context.

With this naming system, one can now use sequences as oracles in
computations. We view the combinatorial classes of sequences as mass
problems. To measure their relative complexity, we compare them via
\emph{Medvedev reducibility} $\le_s$: Let~$\+C$ and~$\+D$ be sets of
functions on~$\NN$, also known as \emph{mass problems}. One says that~$\+C$
is Medvedev reducible to~$\+D$ and writes $\+C\le_s \+D$ if there is a Turing
functional~$\Theta$ such that $\Theta^g \in \+C$ for each $g \in \+D$. Less
formally, one says that functions in~$\+D$ uniformly compute functions
in~$\+C$. We will also refer to the weaker \emph{Muchnik reducibility}:
$\+C\le_w \+D$ if each function in~$\+D$ computes a function in~$\+C$.

With subsequent research in mind, we will set up our framework to apply to
general countable Boolean algebras rather than merely the Boolean algebra of
the computable sets. Throughout, we fix a countable Boolean algebra~$\mathbb
B$ of subsets of~$\NN$ closed under finite differences. Our basic objects
will be sequences of sets in~$\mathbb B$. We will obtain meaningful results
already when we fix a countable Turing ideal~$\+I$ and let~$\mathbb B$ be the
sets with degree in~$\+I$. While we mainly study the case when~$\B$ consists
of the computable sets, in Section~\ref{sec:other_BAs}, we briefly consider
two other cases: the $K$-trivial sets and the primitive recursive sets.

\subsection{The mass problem \boldmath\texorpdfstring{$\+T_\B$}{T\_B} of
maximal towers}

\begin{definition}\label{def:tower}
We say that a sequence $\seq{G_n}\sN n$ of sets in~$\mathbb B$ is a $\mathbb
B$-\emph{tower} if $G_0= \NN$, $G_{n+1} \sub^* G_n$, and $G_n\setminus
G_{n+1}$ is infinite for each~$n$. If~$\mathbb B$ consists of the computable
sets, we use the term \emph{tower of computable sets}.
\end{definition}


\begin{definition} \label{df:assoc_fcn}
We say that a function~$p$ is \emph{associated with} a tower~$G$ if~$p$ is
strictly increasing and $p(n ) \in \bigcap_{i\le n} G_i$ for each~$n$.
\end{definition}

The following fact is elementary.

\begin{fact} \label{fa:assoc fcn}
A tower~$G$ uniformly computes a function~$p$ associated with it.
\end{fact}

\begin{proof}
Let~$\Phi$ be the Turing functional such that $\Phi^G(0) = \min (G_0)$, and
$\Phi^G(n+1)$ is the least number in~$\bigcap_{i\le n+1} G_i$ greater than
$\Phi^G(n)$. This~$\Phi$ establishes the required uniform reduction.
\end{proof}

\begin{definition}
Given a countable Boolean algebra~$\mathbb B$ of sets, the mass problem
$\+T_\mathbb B$ is the class of sets~$G$ such that $\seq {G_n}\sN n$ is a
$\mathbb B$-tower that is \emph{maximal}, i.e., such that for each infinite
set $R \in \mathbb B$, there is~$n$ such that $R\setminus G_n $ is infinite.
\end{definition}

Clearly, being maximal implies that no associated function is computable. In
particular, a maximal tower is never computable. (Note that our notion of
maximality only requires that the tower cannot be extended from below, in
keeping with our set-theoretic analogy.)

\subsection{The mass problem \boldmath\texorpdfstring{$\+U_\B$}{U\_B} of
ultrafilter bases}

We now define the mass problem~$\+U_\B$ corresponding to the ultrafilter
number. Since all filters of our Boolean algebras are countable, any base
will compute a linearly ordered base by taking finite intersection. So for
measuring the relative complexity via Medvedev reducibility, we can restrict
ourselves to linearly ordered bases. Importantly, we require that each
ultrafilter base is a tower; in particular, the difference between a set and
its successor is infinite. (Asking that an ultrafilter base is linearly
ordered is not always possible in the setting of set theory, as discussed in
the introduction.)

\begin{definition} \label{def:UFB}
Given a countable Boolean algebra~$\mathbb B$ of sets, let~$\+U_\B$ be the
class of sets~$F $ such that~$F$ is a $\mathbb B$-tower as in
Definition~\ref{def:tower} and for each set $R\in \mathbb B$, there is~$n$
such that $F_n \sub^* \overline R$ or $F_n \sub^* R$. We will call a set~$F$
in~$\+U_\B$ a $\mathbb B$-\emph{ultrafilter base}.
\end{definition}

\noindent Each ultrafilter base is a maximal tower. In the cardinal setting,
one has $\mathfrak t \le \mathfrak u$. Correspondingly, since $\+U_\B \sub
\+T_\B$, we trivially have $\+T_\B \le_s \+U_\B $ via the identity reduction.
The following indicates that for many natural Boolean algebras, ultrafilter
bases necessarily have computational strength.

\begin{prop}\label{prop:CDno}
Given a Turing ideal~$\+K$, let ~$\mathbb B$ be the Boolean algebra of sets
with degree in~$\+ K$. Then for each $\mathbb B$-ultrafilter base~$F$ and
associated function~$p$ in the sense of Definition~\ref{df:assoc_fcn}, the
function~$p$ is not dominated by a function with Turing degree in~$\+K$.
\end{prop}

\begin{proof}
Assume that there is a function $f \ge p$ in~$\+K$. The conditions {$n_0=0$}
and $n_{k+1} = f(n_k)+1$ define a sequence that is computable from some
oracle in~$\+K$, and for every~$k$ we have that $[n_k, n_{k+1})$ contains an
element of $\bigcap_{i\leq k} F_i$. So the set
\[
E= \bigcup_{i\in\omega}\; [n_{2i}, n_{2i+1})
\]
is in~$\+K$, and clearly $F_n \not \sub^* E$ and $F_n \not \sub^* \ol E$ for
each~$n$. Therefore,~$F$ is not a $\mathbb B$-ultrafilter base.
\end{proof}

\subsection{The Boolean algebra of computable sets}

We finish the introduction by summarizing our results in the case that~$\B$
is the Boolean algebra of all computable sets. By Theorem~\ref{thm: NonLow},
every non-low set computes a set in~$\+T_\B$, and this is uniform. This is
not a characterization, however, because by Corollary~\ref{freaky}, every
noncomputable c.e.\ set computes a maximal tower. On the other hand, we know
that there are noncomputable (necessarily low) sets that do not compute
maximal towers; in particular, no 1-generic $\DII$-set does so. This is
because 1-generic $\DII$-sets are index guessable by
Theorem~\ref{thm:index_predictable}, and by Proposition~\ref{pr:halt}, no
index guessable set can compute a maximal tower. Here, an oracle~$G$ is
\emph{index guessable} if~$\emptyset'$ can find a computable index for
$\phi_e^G$ uniformly in~$e$, provided that~$\phi_e^G$ is computable. We do
not know whether index guessability characterizes the oracles that are unable
to compute a maximal tower. It seems unlikely; index guessability appears to
be stronger than necessary.

As already mentioned, in the setting of cardinal characteristics, $\mathfrak
t < \mathfrak u$ is consistent with ZFC. Since non-low oracles can be
computably dominated, it follows from Proposition~\ref{prop:CDno} that there
is a member of~$\+T_\B$ that does not compute any member of~$\+U_\B$. In
other words, $\+U_\B \not \le_w \+T_\B$ in the case that~$\B$ consists of the
computable sets.

The separation above only uses the fact that members of~$\+U_\B$ are not
computably dominated; in fact, they are \emph{high}. As we show in
Theorems~\ref{thm:UFB->high} and~\ref{thm: High},~$\+U_\B$ is Medvedev
equivalent to the mass problem of dominating functions. In
Section~\ref{sec:4}, we prove that the mass problem~$\+I_\B$ of maximal
independent families is also Medvedev equivalent to the mass problem of
dominating functions. Thus, in the case that~$\B$ is the Boolean algebra of
computable sets, we have $\+U_\B\equiv_s\+I_\B$. Interestingly, we do not
have a direct proof. Contrast this with the equivalence of~$\+T_\B$ and
$\+A_\B$, the mass problem of maximal almost disjoint families; this
equivalence is direct and holds for an arbitrary Boolean algebra, as we will
see presently.

\section{Basics of the mass problems \texorpdfstring{$\+T_\B$}{T\_B}}

\subsection{The equivalent mass problems \boldmath\texorpdfstring{$\+T_\B$
and $\+A_\B$}{T\_B and A\_B}}

Recall that in set theory, the almost disjointness number~$\mathfrak a$ is
the least possible size of a maximal almost disjoint (MAD) family of subsets
of~$\omega$. In our analogous setting, we call a sequence $\seq {F_n}\sN n$
of sets in~$\mathbb B$ \emph{almost disjoint} (AD) if each~$F_n$ is infinite
and $F_n \cap F_k $ is finite for distinct~$n$ and~$k$.

\begin{definition}
In the context of a Boolean algebra~$\mathbb B$ of sets, the mass
problem~$\+A_\mathbb B$ is the class of sets~$F$ such that $\seq {F_n}\sN n$
is a \emph{maximal almost disjoint} (MAD) family in~$\B$. Namely, the
sequence is AD, and for each infinite set $R \in \B$, there is~$n$ such that
$R\cap F_n $ is infinite.
\end{definition}

\begin{fact}\label{fa:ATA}
$\+A_\mathbb B \le_s \+T_\mathbb B \le_s \+A_\mathbb B$.
\end{fact}

\begin{proof}
We suppress the subscript~$\mathbb B$. To check that~$\+A \le_s \+T$, given a
set~$G$, let $\mathrm{Diff}(G) $ be the set~$F$ such that $F_n = G_n\setminus
G_{n+1}$ for each~$n$. Clearly, the operator $\mathrm{Diff} $ can be seen as
a Turing functional. If~$G$ is a maximal $\B$-tower, then $F=\mathrm{Diff}(G)
$ is MAD. For, if~$R\in\B$ is infinite, then $R\setminus G_n$ is infinite for
some~$n$, and hence $R\cap F_i$ is infinite for some $i<n$.

For $\+T \le_s \+A$, given a set~$F$, let $G =\mathrm{Cp}(F) $ be the set
such that
\[
x \in G_n \lra \fa i < n \, [ x \not \in F_n].
\]
Again,~$\mathrm {Cp}$ is a Turing functional. If~$F$ is an almost disjoint
family of sets from~$\B$, then~$G$ is a $\B$-tower, and if~$F$ is MAD,
then~$G$ is a maximal tower.
\end{proof}

Recall that a maximal tower is not computable. Hence no MAD family is
computable. (This corresponds to the cardinal characteristics being
uncountable.)

\subsection{Descriptive complexity and index complexity for maximal towers}

For the rest of this section, as well as the subsequent three sections, we
will mainly be interested in the case that~$\B$ is the Boolean algebra of all
computable sets. We will omit the parameter~$\B$ when we name the mass
problems. In the final section, we will consider other Boolean algebras.

Besides looking at the relative complexity of mass problems such as~$\+T$
and~$\+U$, one can also look at the individual complexity of their members
(as sets encoding sequences). Recall that a characteristic index for a
set~$M$ is a number~$e$ such that $\chi_M= \phi_e$. The following two
questions arise:

\begin{enumerate}
\item\label{Q1}
How low in the arithmetical hierarchy can the set be located?
\item
How hard is it to find characteristic indices for the sequence
\mbox{members}?
\end{enumerate}


\subsubsection*{Arithmetical complexity}

\begin{fact}\label{fa:silly}
No maximal tower~$G$ is c.e., and no MAD set is co-c.e.
\end{fact}

\begin{proof}
For the first statement, note that otherwise, there is a computable
function~$p$ associated with~$G$ in the sense of~\ref{df:assoc_fcn}. The
range of~$p$ extends the tower~$G$, contrary to its maximality.

For the second statement, note that the reduction~$\mathrm{Cp}$ introduced in
the proof of Fact~\ref{fa:ATA} to show that $\+T \le_s \+A$ turns a co-c.e.\
set~$F$ into a c.e.\ set~$G$.
\end{proof}

We will return to Question~\eqref{Q1} in Section~\ref{s:3}, where we show
that c.e.\ MAD sets exist in every nonzero c.e.\ Turing degree, and that some
ultrafilter base is co-c.e.

\subsubsection*{Complexity of finding characteristic indices for the sequence
members}

In several constructions of towers $\seq{G_n}\sN n$ below, such as in
Corollary~\ref{freaky} and Theorem~\ref{th:coce UFB}, the oracle~$\ES''$ is
able to compute, given~$n$, a characteristic index for~$G_n$. The
oracle~$\ES'$ does not suffice by the following result.

\begin{prop}\label{pr:halt}
Suppose that~$G $ is a maximal tower. There is no computation procedure with
oracle~$\Halt $ that computes, from input~$n$, a characteristic index
for~$G_n$.
\end{prop}

\begin{proof}
Assume the contrary. Then there is a computable function~$f$ such that
$\phi_{\lim_s f(n,s)}$ is the characteristic function of~$G_n$. Let~$\hat G$
be defined as follows. Given~$n$ and~$x$, compute the least $s> x$ such that
$\phi_{f(n,s),s}(x) \DA$. If the output is not~$0$, put~$x$ into~$\hat G_n$.
Clearly~$\hat G$ is computable. Since $G_n = ^* \hat G_n$ for each~$n$,~$\hat
G$ is a maximal tower, contrary to Fact~\ref{fa:silly}, or to the earlier
observation that maximal towers cannot be computable.
\end{proof}

\section{\texorpdfstring{Complexity of $\+T$ and of $\+U$}{Complexity of T
and of U}}

In this section, we compare our two principal mass problems, maximal towers
and ultrafilter bases, to well-known benchmark mass problems: non-lowness and
highness. We also define index guessability. No index guessable oracle
computes a maximal tower. We show that every 1-generic $\DII$-set is index
guessable.

As we said above, we restrict ourselves to the case that~$\B$ is the Boolean
algebra of computable sets, and usually drop the subscripts $\B$.

\subsection{Maximal towers, non-lowness, and index guessability}

We now show that each non-low oracle computes a set in~$\+T$. The result is
uniform in the sense of mass problems. Let NonLow denote the class of oracles
$Z$ such that $Z' \not \le_T \ES'$.

\begin{thm}\label{thm: NonLow}
$\+T \le_s \mathrm{NonLow}$.
\end{thm}

\begin{proof}
In the following,~$x$,~$y$, and~$z$ denote binary strings; we identify such a
string~$x$ with the number whose binary expansion is~$1x$. For example, the
string~$000$ is identified with 8, the number with binary representation
$1000$. Define a Turing functional~$\Theta$ for the Medvedev reduction as
follows: Set $\Theta^Z = G$, where for each~$n$,
\[
G_n= \{x \colon \, n \le s:= |x| \land Z'_s \uhr n= x \uhr n\}.
\]
Here~$Z'$ denotes the jump of~$Z$, which is computably enumerated relative
to~$Z$ in a standard way. Note that, for each~$n$, for sufficiently
large~$s$, the string $Z'_s \uhr n$ settles. So it is clear that for
each~$n$, we have $G_{n+1} \sub^* G_n$ and $G_n\setminus G_{n+1}$ is
infinite. Also~$G_n$ is computable.

Suppose now that~$R$ is an infinite set such that $R \sub^* G_n$ for each
$n$. Then for each~$k$,
\[
Z'(k)= \lim_{x \in G_k,|x| > k } x(k) = \lim_{x \in R,|x| > k } x(k),
\]
and hence $Z'\le_T R'$. So if $Z \in \mathrm{NonLow}$, then~$R$ cannot be
computable, and hence $\Theta^Z \in \+T$.
\end{proof}

\begin{remark}
The proof above yields a more general result. Suppose that~$\mathcal K$ is a
countable Turing ideal and~$\mathbb B$ is the Boolean algebra of sets with
degree in~$\mathcal K$. Then $\mathcal T_\mathbb B \le_s
\mathrm{NonLow}_\mathcal K$, where $\mathrm{NonLow}_\mathcal K : = \{ Z
\colon \fa R \in \mathcal K \, [Z' \not \le_T R']\}$.
\end{remark}

We next introduce a property of oracles that we call \emph{index
guessability}; it implies that an oracle does not compute a maximal tower. As
usual, let $\seq{\Phi_e}\sN e$ be an effective list of the Turing functionals
with one input, and write~$\varphi_e $ for $\Phi_e^\ES$. Note that if~$L$ is
a $\DII$-oracle, then~$\emptyset''$ can compute from~$e$ a characteristic
index for~$\Phi_e^L$ in case that the function~$\Phi_e^L$ is computable. To
be index guessable means that~$\Halt $ suffices.

\begin{definition}
We call an oracle~$L$ \emph{index guessable} if~$\Halt$ can compute from~$e$
an index for~$\Phi_e^L $ whenever~$\Phi_e^L$ is a computable function. In
other words, there is a functional~$\Gamma$ such that \bc $\Phi_e^L$ is
computable~$\RA$ $\Phi_e^L = \varphi_{\Gamma(\Halt;e)}$. \ec No assumption is
made on the convergence of $\Gamma(\Halt;e)$ in case~$\Phi_e^L$ is not a
computable function.
\end{definition}


Clearly, being index guessable is closed downward under~$\le_T$. A total
function is computable if and only if its graph is computable, in a uniform
way. So for index guessability of~$L$, it suffices that there is a Turing
functional~$\Gamma$ such that $\Gamma(\Halt;e)$ provides an index
for~$\Phi_e^L$ in case it is a computable $\{0,1\}$-valued function.

Every index guessable oracle~$D$ is low. To see this, for $i \in \NN$, let
$B_i= \{t\colon i\in D'_t\}$. If $i \in D'$ then~$B_i$ is cofinite, otherwise
$B_i= \emptyset$. There is a computable function~$g$ such that $\Phi_{g(i)}^D
$ is the characteristic function of~${B_i}$. To show that $D '\le_T \Halt$,
on input~$i$, let~$\emptyset'$ compute a computable index~$r(i)$ for~$B_i$.
Now use~$\emptyset' $ again to determine $\lim_k \varphi_{r(i)}(k)$, which
equals~$D'(i)$.

By Proposition~\ref{pr:halt}, an index guessable oracle~$D$ does not compute
a maximal tower. The following provides examples of such oracles.

\begin{thm}\label{thm:index_predictable}
If~$L$ is $\Delta^0_2$ and $1$-generic, then~$L$ is index guessable.
\end{thm}

\begin{proof}
Suppose that $F = \Phi_e^L$ and~$F$ is a computable set. Let~$S_e$ be the
c.e.\ set of strings~$\sss$ above which there is a $\Phi_e$-splitting in the
sense that
\[
S_e = \set{\sigma}{(\exists p)(\exists \tau_1\succ\sigma)
  (\exists \tau_2\succ\sigma) \, \Phi_e^{\tau_1}(p)\neq \Phi_e^{\tau_2}(p)}.
\]
Suppose that~$S_e$ is dense along~$L$. Then we claim that the set
\[
C_e = \set{\tau}{ (\exists p) \, \Phi_e^{\tau}(p)\neq F(p)}
\]
is also dense along~$L$, i.e., for every~$k$, there is some $\tau\succeq
L\uhr k$ such that $\tau\in C_e$. Indeed, let $\sigma\succeq L\uhr k$ be a
member of~$S_e$ and let~$p$,~$\tau_1$, and~$\tau_2$ witness this.
Let~$\tau_i$ for $i = 1$ or~$2$ be such that $ \Phi_e^{\tau_i}(p)\neq F(p)$.
Then $\tau_i\succeq L\uhr k$ is in~$C_e$. The set~$C_e$ is c.e.\ and
hence~$L$ meets~$C_e$, contradicting our assumption that $F = \Phi_e^L$.

It follows that~$S_e$ is not dense along~$L$. In other words, there is some
least~$k_e$ such that there is no splitting of~$\Phi_e$ above~$L\uhr k_e$. On
input~$e$, the oracle~$\emptyset'$ can compute~$k_e$ and $L\uhr k_e$. This
allows~$\emptyset '$ to find an index for~$F$, given by the following
procedure: To compute $F(p)$, find the least $\tau\succeq L\uhr k_e$ such
that $\Phi_e^{\tau}(p)\downarrow$ (in~$|\tau|$ many steps). Such a~$\tau$
exists because $\Phi_e^L(p)\downarrow$. By our choice of~$k_e$, it follows
that $\Phi_e^{\tau}(p) = \Phi_e^L(p) = F(p)$.
\end{proof}

We summarize the known implications: \bc 1-generic $\DII$~$\RA$ index
guessable~$\RA$ computes no maximal tower~$\RA$ low. \ec The last implication
cannot be reversed by Theorem~\ref{MAD ce} below; the others might. In
particular, we ask whether any oracle that computes no maximal tower is index
guessable. This would strengthen Theorem~\ref{thm: NonLow}. Note that the
following apparent weakening of index guessability of~$L$ still implies that
the oracle~$L$ computes no maximal tower: For each $S \le_T L$ such that
each~$S^{[n]}$ is computable, there is a functional~$\Gamma$ such that $
\varphi_{\Gamma(\Halt;n)}$ is the characteristic function of $S^{[n]}$. To
see this, assume~$S$ is a maximal tower~$G$. Such an~$S$ contradicts
Proposition~\ref{pr:halt}.


%

\begin{aside}
We pause briefly to mention a potential connection of our topic to
computational learning theory. One says that a class~$S$ of computable
functions is \emph{EX-learnable}
if there is a total Turing machine~$M$ such that $\lim_s M(f\uhr s)$ exists
for each $f\in S$ and is an index for~$f$. For an oracle~$A$, one says
that~$S$ is \emph{EX$[A]$-learnable} if there is an oracle machine~$M$ that
is total for each oracle and such that $\lim_s M^A(f\uhr s)$ exists for each
$f\in S$ and is an index for~$f$. One calls an oracle~$A$ \emph{EX-trivial}
if EX = EX$[A]$. Slaman and Solovay~\cite{Slaman.Solovay:91} showed that~$A$
is EX-trivial if and only if~$A$ is~$\DII$ and has 1-generic degree. This
used an earlier result of Haught~\cite{Haught:86}
that the Turing degrees of the 1-generic $\DII$-sets are closed downward.
\end{aside}

\subsection{Ultrafilter bases and highness}


Let $\text{Tot} =\set{e}{\phi_e\text{ is total}}$. Let $\mathrm{DomFcn}$
denote the mass problem of functions~$h$ that dominate every computable
function and also satisfy $h(s) \ge s$ for all~$s$. Note that a set~$F$ is
high if and only if $\text{Tot} \le_T F'$. To represent highness by a mass
problem in the Medvedev degrees, one can equivalently choose the set of
functions dominating each computable function, or the set of approximations
to $\text{Tot}$, i.e., the $\{0,1\}$-valued binary functions~$f$ such that
$\lim_s f(e,s) = \text{Tot}(e)$. This follows from the next fact; we omit the
standard proof.

\begin{fact} \label{fact:sty}
$\mathrm{DomFcn}$ is Medvedev equivalent to the mass problem of
approximations to \emph{$\text{Tot}= \set{e}{\phi_e\text{ is total}}$}.
\end{fact}

We show that exactly the high oracles compute ultrafilter bases, and that the
reductions are uniform. By Fact~\ref{fact:sty}, it suffices to show that $\+U
\equiv_s \mathrm{DomFcn}$. We will obtain the two Medvedev reductions through
separate theorems, with proofs that are unrelated.

\begin{thm}\label{thm:UFB->high}
Every ultrafilter base uniformly computes a dominating function. In other
words, $\+U \ge_s \mathrm{DomFcn}$.
\end{thm}

Our proof is directly inspired by a proof of Jockusch~\cite[Theorem~1,
(iv)$\implies$(i)]{Jockusch:72*1}, who showed that any family of sets
containing exactly the computable sets must have high degree.

\begin{lemma} \label{lem:sequence PI01}
There is a uniformly computable sequence $P_0, P_1,\dots$\ of non\-empty
$\Pi^0_1$-classes such that for every~$e$,
\begin{itemize}
\item if~$\phi_e$ is total, then~$P_e$ contains a single element, and
\item if~$\phi_e$ is not total, then~$P_e$ contains only bi-immune elements.
\end{itemize}
\end{lemma}

\begin{proof}
Note that each Martin-L\"of (or even Kurtz) random set is bi-immune: For an
infinite computable set~$R$, the class of sets containing~$R$ is a
$\PI1$-null class and hence determines a Kurtz test. A similar fact holds for
the class of sets disjoint from~$R$.

For each~$s$, let~$n_s$ be the largest number such that~$\phi_{e,s}$
converges on $[0,n_s)$. We build the $\Pi^0_1$-class~$P_e$ in stages, where
$P_{e,s}$ is the nonempty clopen set we have before stage~$s$ of the
construction. Let $P_{e,0}=2^\omega$.

\emph{Stage~$0$.} Start constructing~$P_e$ as a nonempty $\Pi^0_1$-class
containing only Martin-L\"of random elements.

\emph{Stage~$s$.} If $n_s = n_{s-1}$, continue the construction that is
currently underway, which will produce a nonempty $\Pi^0_1$-class of random
elements.

On the other hand, if $n_s > n_{s-1}$, fix a string~$\sigma$ such that
$[\sigma]\subseteq P_{e,s}$ and~$|\sigma|>s$. Let $P_{e,s+1} = [\sigma]$. End
the construction that we have been following and start a new construction for
$P_e$, starting at stage $s+1$, as a nonempty $\Pi^0_1$-subclass of
$[\sigma]$ containing only Martin-L\"of random elements.

It is clear that if~$\phi_e$ is total, then~$P_e$ will be a singleton.
Otherwise, there will be a final construction of a nonempty $\Pi^0_1$-class
of randoms which will run without further interruption.
\end{proof}

Of course, when~$P_e$ is a singleton, its lone element must be computable.

\begin{proof}[Proof of Theorem~\ref{thm:UFB->high}]
For any set~$C$, let $S_C = \set{X\in 2^\omega}{C\subseteq X}$. Note that
if~$C$ is computable (or even merely c.e.), then~$S_C$ is a $\Pi^0_1$-class.
Let $Q_e = \set{X}{\ol{X}\in P_e}$ be the $\Pi^0_1$-class of complements of
elements of~$P_e$.

Now let~$F$ be an ultrafilter base. We have that
\begin{align*}
\phi_e\text{ is total} \iff& (\exists i)(\exists n)\;
       [F_i\setminus [0,n] \text{ is a subset of some } \\
	& \qquad\qquad X\in P_e\text{ or its complement}] \\
	\iff& (\exists i)(\exists n)\; [P_e\cap S_{F_i\setminus [0,n]}
    \neq\emptyset \text{ or }Q_e\cap S_{F_i \setminus [0,n]}\neq\emptyset].
\end{align*}
Even though $S_{F_i \setminus [0,n]}$ is a $\PPI$-class, we cannot hope to
compute an index using~$F$. However, $S_{F_i \setminus [0,n]}$ is a
$\PPI[F]$-class uniformly in $i,n$. Using the fact that the nonemptiness of a
$\Pi^0_1[F]$-class is a $\Pi^0_1[F]$-property, we see that $\text{Tot}
=\set{e}{\phi_e\text{ is total}}$ is $\Sigma^0_2[F]$.
Note that the $\Sigma^0_2$-index does not depend on~$F$. Since Tot is
also~$\Pi^0_2$, it is $\Delta^0_2[F]$ via a fixed pair of indices, and hence
Turing reducible to~$F'$ via a fixed reduction. One direction of the usual
proof of the (relativized) Limit Lemma now shows that we can uniformly
compute an approximation to $\text{Tot}$ from~$F$. Hence, from~$F$ we can
uniformly compute a dominating function by Fact~\ref{fact:sty}.
\end{proof}

%
%
%

\begin{thm}\label{thm: High}
Every dominating function uniformly computes an ultrafilter base. In other
words, $\+U \le_s \mathrm{DomFcn}$.
\end{thm}

\begin{proof}
Let $\seq{\psi_e}\sN e$ be an effective listing of the $\{0,1\}$-valued
partial computable functions defined on an initial segment of~$\NN$. Let
$V_{e,k} =\{x\colon \psi_e(x)=k\}$ so that $\seq{(V_{e,0}, V_{e,1})}$ is an
effective listing that contains all pairs of computable sets and their
complements.

Let $T=\{0,1,2\}^{< \omega}$.
Uniformly in $\alpha \in T$, we will define a set~$S_\aaa$.
%
%
We first explain the basic idea and then modify it to make it work. The basic
idea is to
start with $S_{\emptyset} = \omega$ and build $S_{\aaa\ape k} = S_\aaa \cap
V_{e,k}$ for $k= 0,1$ and $e = |\alpha|$,
that is, we split~$S_\aaa$ according to the listing above. We
then consider the leftmost path~$g$
such that $S_{g\uhr e}$ is infinite for each~$e$. A~dominating
function~$h$ can eventually discover each initial segment of this path, and
use this to compute a set~$F$ such that $F_e = ^*S_{g\uhr e}$ for each~$e$.

The problem is that both $S_\aaa \cap V_{e,0}$ and $S_\aaa \cap V_{e,1}$
could be finite (because~$e$ is not a proper index of a computable set). In
this case we still need to make sure that $F_e \setminus F_{e+1}$ is
infinite. So the rightmost option at level~$n$ is a set $S_{\aaa \ape 2 }=
\wt S_\aaa$ that simply removes every other element from~$S_\aaa$ (so as to
obtain an infinite coinfinite subset).
The sets $S_{\aaa\ape k}$ for $k \le 1$ will be subsets of~$\wt S_{\aaa}$.

We now provide the details. The set~$S_\aaa$ is enumerated in increasing
fashion, and possibly finite. So each~$S_\aaa$ is computable, though not
uniformly in~$\aaa$. All the sets and functions defined below can be
interpreted at stages.

Let $S_{\ES,s}= [0,s)$. If we have defined (at stage~$s$) the set $S_\aaa =
\{r_0< \cdots <r_k\}$, let $\wt S_\aaa$ contain the numbers of the form
$r_{2i}$. Let $S_{\aaa \ape 2 } = \wt S_\aaa$.
Let $S_{\aaa\ape k} = \wt S_\aaa \cap V_{e,k}$ for $k= 0,1$, $e = |\alpha|$.
We define a uniform list of Turing functionals~$\Gamma_e$ so that the
sequence $\seq {\Gamma^h_e(t)}\sN t$ is nondecreasing and unbounded, for
each~$e$ and each oracle function~$h$ such that $h(s) \ge s$ for each~$s$. We
will let $F_e = \{\Gamma^h_e(t)\colon t \in \NN\}$.
\medskip

\n \emph{Definition of\/~$\Gamma_e$.} Given an oracle function~$h$, we will
write~$a_s$ for $\Gamma^h_e(s)$. Let $a_0=0$. Suppose $s>0$ and~$a_{s-1}$ has
been defined. Check if there is $\aaa \in T$ of length~$e$ such that
$|S_{\aaa, h(s)}|\ge s$. If there is no such~$\aaa$, let $a_s=a_{s-1}$.
Otherwise, let~$\aaa$ be leftmost such. If $\max S_{\aaa, h(s)} > a_{s-1}$,
let $a_s= \max S_{\aaa, h(s)} $. Otherwise, again let $a_s=a_{s-1}$.

Note that the sequence $\{a_s\}_{s<\omega}$ is unbounded because for the
rightmost string $\alpha \in T$ of length~$e$ (i.e., the string consisting
only of~$2$'s), the set~$S_{\alpha,t}$ consists of the numbers in $[0,t)$
divisible by~$\tp{e}$. We may combine the functionals~$\Gamma_e$ to obtain a
functional~$\Psi$ such that $(\Psi^h)_e = F_e$ for each~$h$ with $h(s) \ge s$
for each~$s$.

\begin{claim}
If $h\in \mathrm{DomFcn}$, then $F=\Psi^h \in \+U$.
\end{claim}
\noindent To verify this, let $g \in \tp \omega$ denote the leftmost path in
$\{0,1,2\}^\omega$ such that the set $S_{g \uhr e}$ is infinite for
every~$e$. Note that~$g$ is an infinite path, {because for every~$\alpha$, if
the set~$S_\aaa$ is infinite then so is $S_{\aaa \ape 2}$}.

Fix~$e$ and let $\aaa = g\uhr e$. Let $p(s)$ be the least stage~$t$ such that
$S_{\aaa,t}$ has at least~$s$ elements. Since~$h$ dominates the computable
function~$p$, we will eventually always pick~$\aaa$ in the definition of
$a_s= \Gamma^h_e(s)$. Hence $F_e=^* S_\aaa$. This implies that~$F_e$ is
computable and $F_{e+1}\subseteq^*F_e$.
Clearly, {if~$S_{\aaa}$ is infinite, then} $S_\aaa \setminus S_\beta$ is
infinite for every $\beta\succ\aaa$. Thus $F_{e} \setminus F_{e+1}$ is
infinite.

Now let~$R$ be a computable set. Pick~$e$ such that $R= V_{e, 0}$ and $\ol R=
V_{e,1}$.
If $g(e)=0$, then $S_{g\, \uhr {e+1}} \sub V_{e,0}$ and hence $F_{e+1} \sub^*
R$. Otherwise, $S_{g\, \uhr {e+1}} \sub V_{e,1}$ and hence $F_{e+1} \sub^*
\ol R$.
\end{proof}

\section{Maximal independent families in computability}
\label{sec:4}

In this short section, we determine the complexity of the
computability-theoretic analog of the independence number~$\mathfrak i$ for
the Boolean algebra of computable sets. It~turns out that in the context of
the computable sets, \emph{maximal independent families} behave in a way
similar to ultrafilter bases.

Given a sequence $\seq {F_n}\sN n$, let $F_\ES = \NN$; for each nonempty
binary string~$\sss$ we write
\begin{equation} \label{eqn:Fsigma}
F_\sss = \bigcap_{\sss(i)= 1} F_i \cap \bigcap_{\sss(i)=0} \ol F_i.
\end{equation}
We call (a set~$F$ encoding) such a sequence \emph{independent} if each set
$F_\sss$ is infinite.

\begin{definition}
Given a Boolean algebra of sets~$\mathbb B$, the mass problem $\+I_\mathbb B$
is the class of sets~$F$ such that $\seq {F_n}\sN n$ is a family that is
\emph{maximal independent}, namely, it is independent, and for each set $R
\in \mathbb B$, there is~$\sss$ such that $F_\sss \sub^* R $ or $F_\sss\sub^*
\ol R$.
\end{definition}

In the following, we let~$\mathbb B$ be the Boolean algebra of computable
sets, and we drop the parameter~$\mathbb B$ as usual. An easy modification of
the proof of Theorem~\ref{thm:UFB->high} yields the following

\begin{thm}\label{thm:MIF->high}
Every maximal independent family~$F$ uniformly computes a dominating
function. In other words, $\+I \ge_s \mathrm{DomFcn}$.
\end{thm}

\begin{proof}
Define the $\PI 1$-classes~$P_e$ as in Lemma~\ref{lem:sequence PI01}. As
before let $Q_e = \set{X}{\ol{X}\in P_e}$ be the $\Pi^0_1$-class of
complements of elements of~$P_e$. Recall that for any set~$C$, we let $S_C =
\set{X\in 2^\omega}{C\subseteq X}$. Now we have that
\begin{align*}
\phi_e\text{ is total} \iff& (\exists \sss)(\exists n)\;
    [F_\sss\setminus [0,n] \text{ is a subset of some } \\
	& \qquad\qquad X\in P_e \text{ or its complement}] \\
	\iff& (\exists \sss )(\exists n)\; [P_e\cap S_{F_\sss\setminus
    [0,n]}\neq\emptyset \text{ or }Q_e\cap S_{F_\sss\setminus [0,n]}
    \neq\emptyset]
\end{align*}
As before, this shows that from~$F$ one can uniformly compute a dominating
function.
\end{proof}

\begin{thm}\label{thm: HighMief}
Every dominating function~$h$ uniformly computes a maximal independent
family. In other words, $\+I \le_s \mathrm{DomFcn}$.
\end{thm}

\n In fact, we will prove that a dominating function~$h$ uniformly computes a
set~$F$ such that the $=^*$-equivalence classes of the sets~$F_e$ freely
generate the Boolean algebra of computable sets modulo finite sets. This
clearly implies that~$F$ is maximal independent: If~$R$ is an infinite
computable set, then for some~$e$ and nonempty set~$S$ of strings of
length~$e$, one has $R=^* \bigcup_{\sss \in S} F_\sss$, and hence
$F_\sss\sub^* R$ for some~$\sss$.

\begin{proof}
As in the proof of Theorem~\ref{thm: High}, let $\seq{\psi_e}\sN e$ be an
effective listing of the $\{0,1\}$-valued partial computable functions
defined on an initial segment of~$\NN$, and let $V_{e,k} =\{x\colon
\psi_e(x)=k\}$ for $k = 0,1$.

In Phase~$e$ of the construction, we will define a computable set~$F_e$ such
that $F_e= \Theta_e^h$ for a Turing functional~$\Theta_e$ determined
uniformly in~$e$. Suppose we have defined~$\Theta_i$ for $i< e$, and thereby
have defined the sets~$F_\sss$ given by~\eqref{eqn:Fsigma} for each
string~$\sss$ of length~$e$.

The idea for building~$F_e$ is to attempt to follow~$V_{e,0}$ while
maintaining independence from the previous sets. We apply this strategy
separately on each~$F_\sss$. Using~$h$ as an oracle we compute recursively an
increasing sequence $ \seq{r^e_n}\sN n$. We carry out the attempts on
intervals $[r^e_n, r^e_{n+1})$. If~$V_{e,0}$ appears to split~$F_\sss$ on the
current interval, then we follow it; otherwise, we merely make sure
that~$F_e$ remains independent from~$F_\sss$ on the interval by putting one
number in and leaving another one out. To decide which case holds, we consult
the dominating function $h$ as an oracle.

We now provide the details for Phase~$e$. Let $r^e_0=0$. If~$r^e_n$ has been
defined, let $r^e_{n+1}> r^e_n$ be the least number~$r$ such that for
each~$\sss$ of length~$e$, the following two conditions hold: \bi
\item[(a)$_\sss$]
$|[r^e_n, r) \cap F_\sss | \ge 2$;
\item[(b)$_\sss$]
if there are $u, w \in \dom (\psi_{e, h(r^e_n)})\cap F_\sss$ with $r^e_n \le
u< w$ such that $\psi_e(u)= 1 $ and $\psi_e(w) =0$, then $r>w$ for the least
such~$w$. \ei

\n We define $F_e(x) = \Theta_e^h(x)$ for $x\in [r^e_n, r^e_{n+1})$ as
follows. Let~$\sss$ be the string of length~$e$ such that $x \in F_\sss$.

\bi
\item
If the hypothesis of condition~(b)$_\sss$ holds and~$\psi_e$ is defined on
$[r^e_n, r^e_{n+1})$, then let $F_e(x)= \psi_e(x)$;
\item
otherwise, if $x=\min ([r^e_n, r^e_{n+1}) \cap F_\sss)$, let $F_e(x) = 1$,
else let $F_e(x)=0$. \ei
\medskip

\n \emph{Verification.} By induction on~$e$, one verifies that for each
function~$h$, the set~$F_\sss$ is infinite for each~$\sss$ with $\sssl = e$,
and that the sequence $\seq{r^e_n}\sN n$ defined in Phase~$e$ of the
construction is infinite. Thus~$\Theta^h_e$ is total. So $F \le_T h$ where
$F_e=\Theta^h_e$, and~$F$ is an independent family.

\begin{claim}
Each set~$F_e$ is computable.
\end{claim}

\n We verify this by induction on~$e$. Suppose it holds for each $i<e$. So
$F_\sss$ is computable for $\sssl =e$.

First assume that~$\dom (\psi_e)$ is finite. Then for sufficiently large~$n$,
condition~(b)$_\sss$ does not apply to any string $\sss$ of length~$e$, and
so the sequence $\seq {r^e_n}\sN n$ is computable. Hence~$F_e$ is computable.

Now assume that~$\psi_e$ is total. Let \[ D_e = \{\sss \colon \, \sssl = e
\land |F_\sss \cap V_{e,0} | = |F_\sss \cap V_{e,1}| = \infty\}. \] Define a
function~$p$ by letting $p(m)$ be the least stage~$s$ such that for each
$\sss \not \in D_e$, condition~(a)$_\sss$ holds with $r^e_n= m$ and $r=s$,
and for each $\sss \in D_e$, there are $u,w \in \dom ({\psi_{e,s}})$ such
that $m \le u<w $ as in the hypothesis of condition (b)$_\sss$. (Let $p(m) =
0$ if~$m$ is not of the form~$r^e_n$.) Since~$F_\sss$ is computable for
each~$\sss$ of length~$e$, the function~$p$ is computable. Since~$h$
dominates~$p$, for sufficiently large~$n$, we will define~$r^e_{n+1}$ by
checking the convergence of computations~$\psi_e(z)$ at a stage $h(r^e_n) \ge
p(r^e_n)$; since in Phase~$e$ of the construction, we chose the witnesses
minimal, $r^e_{n+1}$ is determined by stage $p(r^e_n)$. So we might as well
check the convergence of computations~$\psi_e(z)$ at stage $p(r^e_n)$. Hence
again, the sequence $\seq {r^e_n}\sN n$ is computable.

\begin{claim}
Suppose that~$\psi_e$ is total. Then for each string $\tau = \sss \ape a$ of
length $e+1$, $F_\tau \sub ^* V_{e,0} $ or $F_\tau \cap V_{e,0} =^* \ES$ (so
that $V_{e,0}=^*\bigcup_\tau \{ F_\tau \colon F_\tau \sub ^* V_{e,0}\}$) .
\end{claim}

\n Let $D_e$ be as above. If $\sss \not \in D_e$, then this is immediate since
$F_\sss \sub^* V_{e,i}$ for some~$i$. Otherwise, Phase~$e$ of the
construction ensures that $F_{\sss \ape 0} =^* F_\sss \cap V_{e,0}$.

By the last claim, the $=^*$-equivalence classes of the~$F_e$ freely generate
the Boolean algebra of the computable sets modulo finite sets. In
particular,~$F$ is a maximal independent family.
\end{proof}

As mentioned in the introduction, we do not know at present whether there is
a ``natural'' Medvedev equivalence between the two mass problems~$\+U$ and
$\+I$ as is the case for~$\+A$ and~$\+T$. This would require direct proofs
avoiding the detour via the mass problem of dominating functions. For what it
is worth, the cardinal characteristics~$\mathfrak{u}$ and~$\mathfrak{i}$ are
incomparable (i.e., ZFC cannot determine their order).

\section{The case of computably enumerable complements} \label{s:3}

Recall from Fact~\ref{fa:silly} that no maximal tower, and in particular no
ultrafilter base, can be computably enumerable. In contrast, in this section
we will see that even ultrafilter bases can have computably enumerable
complement. As in the previous sections, we are restricting our attention to
the Boolean algebra of all computable sets.

Recall that a coinfinite c.e.\ set~$A$ is called \emph{simple} if it meets
every infinite c.e.\ (or, equivalently, every computable) set;~$A$ is called
\emph{$r$-maximal} if $\ol A \sub^* \ol R$ or $\ol A \sub^* R$ for each
computable set~$R$. Each $r$-maximal set is simple. For more background, see
Soare~\cite{Soare:87}.

\subsection{Computably enumerable MAD sets, and co-c.e.\ towers}

We will show that if~$A$ is a noncomputable c.e.\ set, then there is a
co-c.e.\ maximal tower~$G\leq_T A$. Given that it is more standard to build
c.e.\ rather than co-c.e.\ sets, it will be convenient to first build a c.e.\
MAD set $F\le_T A$
and then use the Medvedev
reduction in Fact~\ref{fa:ATA} to obtain a co-c.e.\ maximal tower. We employ a
priority construction with requirements that act only finitely often.


\begin{thm}\label{MAD ce}
For each noncomputable c.e.\ set~$A$, there is a MAD c.e.\ set $F \le_T A$.
\end{thm}

\begin{proof}
The construction is akin to Post's construction of a simple set. In
particular, it is compatible with permitting.

Let $\seq {M_e} \sN e$ be a uniformly c.e.\ sequence of sets such that
$M_{2e} = W_e$ and $M_{2e+1} = \NN$ for each~$e$. We will build an auxiliary
c.e.\ set $H \le_T A$ and let the c.e.\ set $F\le_T A $ be defined by
$F^{[e]} = H^{[2e]} \cup H^{[2e+1]}$. The purpose of the sets~$M_{2e+1}$ is
to make the sets~$H^{[2e+1]}$, and hence the sets~$F^{[e]}$, infinite. The
construction also ensures that~$H$, and hence~$F$, is AD, and that $\bigcup_n
H^{[n]}$ is coinfinite.

As usual, we will write~$H_e$ for~$H^{[e]}$.
We provide a stage-by-stage construction to meet the requirements
\[
P_n \colon M_e\setminus \bigcup_{i<n} H_i \text { infinite } \Rightarrow
   |H_e \cap M_e | \ge k \text{, where } n= \langle e, k \rangle.
\]
(Note that the union is over all~$i$ such that $i<n$, not $i<e$.) At
stage~$s$, we say that~$P_n$ is \emph{ permanently satisfied} if $|H_{e,s}
\cap M_{e,s}| \ge k$.

\medskip
\noindent \emph{Construction.}

\nopagebreak \noindent \emph{Stage $s>0$}. See if there is $n< s$ such
that~$P_n$ is not permanently satisfied, and, where $ n= \langle e, k
\rangle$, there is $x \in M_{e,s}\setminus \bigcup_{i<n} H_{i,s}$ such that
\bc $x > \max (H_{e,s-1})$, $x \ge 2n$, and $A_s \uhr x \neq A_{s-1} \uhr x$.
\ec If so, choose~$n$ least, and put $\langle x, e \rangle$ into~$H$ (i.e.,
put~$x$ into~$H_e$).
\medskip

\noindent \emph{Verification.} Each~$H_e$ is enumerated in increasing fashion
and hence computable.

Each~$P_n$ is active at most once.
This ensures that $\bigcup_e H_e$ is coinfinite: For each~$N$, if $x< 2N$
enters this union, then this is due to the action of a requirement~$P_n$ with
$n < N$, so there are at most~$N$ many such~$x$.

To see that a requirement~$P_n$ for $ n= \langle e, k \rangle$ is met,
suppose that its hypothesis holds. Then there are potentially infinitely many
candidates~$x$ that can go into~$H_e$. Since~$A$ is noncomputable, one of
them will be permitted.

Now, by the choice of $M_{2e+1}$ and the fact that $\bigcup_e H_e$ is
coinfinite, each $H_{2e+1}$, and hence each~$F_e$, is infinite. We claim that
for $e< m$, we have $|H_e \cap H_m| \le m$. For suppose that $x \in H_m$
enters $H_e$ at stage $s$. Then $x \in H_{m,s}$ since $r \ge \la m,0\ra > e$
for any requirement $P_r$ putting $x$ into $H_m$. Suppose $P_n$ puts $x$ into
$H_e$ at stage $s$, where $n= \la e,k\ra$. Then $n \le m$, so the claim
follows as each requirement is active at most once. We conclude that the
family described by~$H$, and therefore also the one described by~$F$, is
almost disjoint.

To show that~$F$ is MAD, it suffices to verify that if~$M_e$ is infinite then
$H_p \cap M_e $ is infinite for some~$p$.  If all the~$P_{\la e,k\ra}$ are
satisfied during the construction, we let $p=e$. Otherwise, we let~$k$ be
least such that~$P_{n}$ is never satisfied where $n= \langle e,k \rangle$.
Then its hypothesis fails, so $M_e \sub^* \bigcup_{i<n} H_i $. Hence $H_p
\cap M_e $ is infinite for some~$p < n$ by the pigeonhole principle.
\end{proof}



Since an index guessable set computes no MAD set by
Proposition~\ref{pr:halt}, we obtain the following

\begin{cor}
If a c.e.\ set~$L$ is index guessable, then $L$ is computable.
\end{cor}

Downey and Nies have given a direct proof of this fact;
see~\cite{LogicBlog:19}.

\begin{cor}\label{freaky}
For each noncomputable c.e.\ set~$A$, there is a co-c.e.\ set $G \le_T A$
such that $G \in \+T$, i.e., $\seq{G_n}\sN n$ is a maximal tower.
\end{cor}

\begin{proof}
Let~$F$ be the MAD set obtained above. Recall the Turing reduction
$\mathrm{Cp}$ showing that $\+T \le_s \+A$ in Fact~\ref{fa:ATA}. The set $G
=\mathrm{Cp}(F)$, given by
\[
x \in G_n \lra \fa i < n \, [ x \not \in F_n]
\]
is as required.
\end{proof}

\subsection{Co-c.e.\ ultrafilter bases}
%



We next construct a co-c.e.\ ultrafilter base~$F$ for the Boolean algebra of
computable sets. That is,~$F$ is co-c.e., each~$F_e$ is computable (but not
uniformly so), and~$F$ is a tower satisfying the condition in
Definition~\ref{def:UFB}.

\begin{thm} \label{th:coce UFB}
There is a co-c.e.\ ultrafilter base~$F$.
\end{thm}


%

\begin{proof}
We adapt the construction from the proof of the main result in
\cite{Lempp.Nies.etal:01}, which states that there is an $r$-maximal set~$A$
such that the index set $\mathrm{Cof}_A= \{e \colon W_e \cup A =^* \NN\}$ is
$\SI 3$-complete. Both the original and the adapted version make use of the
fact that we are given a c.e.\ index for a computable set and also one for
its complement (see the pairs $(V_{e,0},V_{e,1})$ below). Our proof can also
be viewed as a variation on the proof of Theorem~\ref{thm: High} in the
setting of co-c.e.\ sets. We remark that by standard methods, one can extend
the present construction to include permitting below a given high c.e.\ set.

We build a co-c.e.\ tower~$F$ by providing uniformly co-c.e.\ sets~$F_e$ for
$e \in \NN$ that form a descending sequence with $F_e \supseteq F_{e+1}$. We
achieve the latter condition by agreeing that whenever we remove~$x$
from~$F_e$ at a stage~$s$, we also remove it from all~$F_i$ for $i>e$.
Furthermore, no element is ever removed from~$F_0$, so $F_0= \NN$.

Let $\seq{(V_{e,0}, V_{e,1})}_{\sN e}$ be an effective listing of all pairs
of disjoint c.e.\ sets as defined in the proof of Theorem~\ref{thm: High}.
The construction will ensure that the following requirements are met:
\begin{align*}
M_e &\colon F_e \setminus F_{e+1}\text{ is infinite,}\\
P_e &\colon V_{e,0} \cup V_{e,1} = \NN \RA F_{e+1} \sub^* V_{e,0} \vee F_{e+1}
   \sub^* V_{e,1}.
\end{align*}
This suffices to establish that~$F$ is an ultrafilter base.

The tree of strategies is $T=\{0,1,2\}^{< \omega}$. Each string $\aaa \in T$
of length~$e$ is tied to~$M_e$ and also to ~$P_e$. We write $\aaa \colon M_e$
and $\aaa \colon P_e$ to indicate that we view~$\aaa$ as a strategy of the
respective type.
\medskip

\n \emph{Streaming.} For each string $\alpha\in T$ with $|\aaa| = e$, at each
stage of the construction, we have a computable  set~$S_\alpha$, thought of
as a stream of numbers used by~$\aaa$.   The purpose of the sets $S_\alpha$
is twofold:
\bi
\item[(a)] to be able to provide candidates for $P_e$ by a
procedure of reserving numbers from the stream, and processing them making
use of its  hypothesis,  and
\item[(b)] to show that $F_e$ is computable.
\ei For~(b), in Claim~\ref{cl:3} we will verify that $F_e=^* S_\aaa$
where~$\aaa$ is the string of length~$e$ on the true path. Since the true
path is merely computable in~$\emptyset''$, we cannot directly define the
co-c.e.\ set~$F$ using the~$S_\aaa$. Rather, we need to spread the
construction of the~$F_e$ over the whole $e$-th level of the tree of
strategies.

We provide some more detail on the dynamics of the streams.
Each time~$\alpha$ is initialized,~$S_\alpha$ is removed from $F_{e+1}$,
and~$S_\alpha$ is reset to be empty. Also,~$S_\alpha$ is enlarged only at
stages at which~$\alpha$ appears to be on the true path.

We will verify the
following conditions on  the final versions of the $S_\aaa$:
\begin{enumerate}
\item $S_\emptyset = \NN$;
\item if~$\alpha$ is not the empty node, then~$S_\alpha$ is a subset of
$S_{\alpha^-}$ (where~$\alpha^-$ is the immediate predecessor of~$\alpha$);
\item at every stage, $S_\gamma \cap S_\beta = \ES$ for incomparable
strings~$\gamma$ and~$\beta$;
\item  any number $x$ is in~$F_{e+1}$ at the time  it first enters~$S_\alpha$;
\item if~$\alpha$ is along the true path of the construction, then~$S_\alpha$
is an infinite computable set.
\end{enumerate}
Note that~$S_\alpha$ is d.c.e.\ uniformly in~$\alpha$. The set~$S_\alpha$ is
finite if~$\alpha$ is to the left of the true path of the
construction;~$S_\alpha$ is an infinite computable set if~$\alpha$ is along
the true path; and~$S_\alpha$ is empty if~$\alpha$ is to the right of the
true path.
\medskip

The {intuitive strategy $\aaa \colon P_e$} is as follows. Only strategies
associated with a string of length~$\le e$ can remove numbers from~$F_{e+1}$.
A~strategy $\aaa \colon P_e$ removes elements from~$S_\aaa$, and at the same
time from~$F_{e+1}$. It regards the set of remaining numbers as its own
version of~$F_{e+1}$; if~$\aaa$ is on the true path then this version is the
true~$F_{e+1}$ up to~$=^*$, as mentioned above. The strategy has to make sure
that no strategies~$\beta$ to its right remove numbers from~$F_{e+1}$ that it
wants to keep. On the other hand, it can only process a number~$x$ once it
knows whether~$x$ is in $V_{e,0}$ or~$V_{e,1}$. The solution to this conflict
is that~$\aaa$ \emph{reserves} a number~$x$ from the stream~$S_{\alpha}$,
which,  by an initialization $\aaa$ carries out at this stage, withholds it
from any action of such a~$\beta$. It then waits until all numbers~$\le x$
are in $V_{e,0}\cup V_{e,1}$. If that never happens for some reserved~$x$,
then~$\aaa$ is satisfied finitarily with eventual outcome~$2$. Otherwise, it
will eventually process~$x$: If $x \in V_{e,0}$, it continues its attempt to
build~$F_{e+1}$ inside~$V_{e,0}$; else it continues to build~$F_{e+1}$
inside~$V_{e,1}$. It takes outcome~$0$ or~$1$, respectively, according to
which case applies. Each time the apparent outcome is~$0$, then the current
$S_{\aaa\ape 1}$ (i.e., the content of its  output stream based on the
assumption that the true outcome is~$1$) is removed from~$F_{e+1}$. So if~$0$
is the true outcome, then indeed $F_{e+1}\sub^* V_{e,0}$; and if~$1$ is the
true outcome, then indeed $F_{e+1}\sub^* V_{e,1}$.
%
\medskip

The \emph{intuitive strategy $\aaa \colon M_e$} simply removes every other
element of~$S_\aaa$ from $F_{e+1}$. Then $\aaa:P_e$ actually only works with
the stream of remaining numbers. There is no further interaction between the
two types of strategies. (Note here that making~$F_{e+1}$ smaller is to the
advantage of~$P_e$.)	Recall that if~$\alpha$ is initialized,~$S_\alpha$ is
removed from $F_{e+1}$, and~$S_\alpha$ is reset to be empty.
\medskip

\n\emph{Construction.}

Stage~$0$. Let~$\delta_0$ be the empty string. Let $F_{e}= \NN$ for each~$e$.
Initialize all strategies.

\emph{Stage}~$s >0$. Let $S_{\ES,s}= [0,s)$. Stage~$s$ consists of substages
$e=0, \ldots ,s-1$, during which we inductively define~$\delta_s$, a string
of length~$s$.
\medskip

\emph{Substage~$e$.} We suppose that $\alpha= \delta_{s}\uhr e$ and
$S_{\alpha}$ have been defined.
\medskip

The strategy $\aaa \colon M_e$ acts as follows. If at the current stage
$S_\aaa = \{r_0< \cdots <r_k\}$ and~$r_k$ is new in~$S_\aaa$, it puts~$r_k$
into $\wt S_\aaa$ if and only if~$k$ is even; otherwise,~$r_k$ is removed
from~$F_{e+1}$.
\medskip

The strategy $\aaa \colon P_e$ picks the first applicable case below.

\textit{Case~1:} Each reserved number of~$\aaa$ has been processed: If there
is a number~$x$ from $\wt S_\aaa$ greater than~$\aaa$'s last reserved number
(if any) and greater than the last stage at which~$\aaa$ was initialized,
pick~$x$ least and \emph{reserve} it. Note that $x<s$ since by definition
$S_{\ES,s}= [0,s)$. Initialize all strategies $\gamma \succeq \aaa\ape 2$,
and let $\aaa\ape 2$ be eligible to act next.

If Case 1 does not apply then~$\aaa$ has a unique reserved, but
unprocessed number~$x$.

\textit{Case~2:} $[0,x] \sub V_{e,0} \cup V_{e,1}$ and $x \in V_{e,0}$:
Let~$t$ be the greatest stage $<s$ at which~$\alpha$ was initialized. Add~$x$
to~$S_{\aaa\ape 0}$ and remove from~$F_{e+1}$ all numbers in the interval
$(t,x)$ that are not in $S_{\aaa\ape 0}$. Declare that~$\aaa$ has
\emph{processed}~$x$. Let ${\aaa \ape 0}$ be eligible to act next.

\textit{Case~3:} $[0,x] \sub V_{e,0} \cup V_{e,1}$ and $x \in V_{e,1}$:
Let~$t$ be the greatest stage $<s$ at which~$\alpha$ was initialized or
$\aaa\ape 0$ was eligible to act. Add~$x$ to $S_{\aaa\ape 1}$ and remove from
$F_{e+1}$ all numbers in the interval $(t,x)$ that are not in $S_{\aaa\ape
1}$. Declare that~$\aaa$ has \emph{processed}~$x$. Let ${\aaa \ape 1}$ be
eligible to act next.

\textit{Case~4:} Otherwise, that is, $[0,x] \not \sub V_{e,0} \cup V_{e,1}$:
Let~$t$ be the greatest stage $<s$ at which~$\alpha$ was initialized, or
$\aaa \ape 0$ or $\aaa \ape 1$ was eligible to act. Let $S_{\aaa\ape 2}= \wt
S_\aaa \cap (t, s)$.
Let $\aaa\ape 2$ be eligible to act.

We define $\delta_{s}(e)=i$ where $\aaa \ape i$, $0 \le i \le 2$, has been
declared eligible to act next. If $e+1< s$, then carry out the next substage.
Else initialize all the strategies~$\beta $ such that $\delta_s <_L \beta$
and end stage~$s$.
\medskip

\n \emph{Verification.} By construction and our convention above,~$F_e$ is
co-c.e., and $F_e \supseteq F_{e+1}$ for each~$e$.

Let $g \in 3^\omega$
denote the true path, namely, the leftmost path in $\{0,1,2\}^\omega$ such
that $\forall e\; \ex^\infty s \, [ g\uhr e \preceq \delta_s]$. In the
following, given~$e$, let $\aaa = g \uhr e$. We verify a number of claims.

\begin{claim}
$\aaa$ is only initialized finitely often.
\end{claim}

\n To see this, let $s_0>0$ be a stage such that $\aaa \le_L \delta_s$ for
each $s \ge s_0$. Suppose the strategy~$\aaa$ is initialized at stage $s \ge
s_0$. Then $\aaa \succeq \beta \ape 2$  for a strategy $\beta\colon P_i$,
where $i= |\beta|$, and  this initialization occurs at Case~1 of substage~$i$
of stage~$s$, namely, when the strategy~$\beta$ reserves a new number~$y$.
However,~$\aaa$ can only be initialized once in that way for each
such~$\beta$: If $\beta$ processes~$y$ at a later stage~$t$, then this causes
$\delta_t <_L \aaa$, contrary to the choice of~$s_0$. This shows the claim.

Let $s_\aaa$ be the largest stage~$s$ such that~$\aaa$ is initialized at
stage~$s$. Note that $\aaa \preceq \gamma$ implies $s_\aaa \le s_\gamma$.

\begin{claim}
The conditions (1)--(5) related to streaming hold.
\end{claim}

(1),~(2) and (4)  hold by construction. (3) Assume this fails for
incomparable~$\gamma$ and~$\beta$, so $x \in S_\gamma \cap S_\beta$ at
stage~$s$. By~(2), we may as well assume that $\gamma= \aaa \ape i$ and
$\beta = \aaa \ape k$ where $i<k$. By construction, $k\le 1$ is not possible,
so $k=2$. Since $x \in S_{\aaa\ape i}$ and $i \le 1$,~$x$ was reserved
by~$\aaa$ at some stage $t\le s$. So~$x$ can never enter $S_{\aaa \ape 2}$ by
the initialization of $\aaa\ape 2$ when~$x$ was reserved by the strategy
$\aaa \colon P_e$ in its Case~1.

(5) holds inductively, by the definition of the true path and
because~$S_\aaa$ is enumerated in increasing fashion at stages $\ge s_\aaa$.

\begin{claim}\label{cl:3}
$F_e =^* S_\aaa$ (and hence, $F_e$ is computable).
\end{claim}

The claim is verified by induction on~$e$. We show that for all $x>
s_{\alpha}$, we have $x\in F_e$ if and only if $x\in S_{\aaa}$. This holds
for $e=0$ because $F_0= S_\ES= \NN$. For the inductive step, let $\gamma = g
\uhr(e+1)$.

First, we verify that $F_{e+1}\cap (s_{\gamma}, \infty)\sub S_\gamma$.
Suppose that $x>s_{\gamma}$ and $x \in F_{e+1}$. Then $x \in F_e$ and
$x>s_{\alpha}$, so by the inductive hypothesis $x \in S_\aaa$. By
construction, any element~$x$ that  does not enter~$S_\gamma$ is also
removed from~$F_{e+1}$ unless~$x$ is the last element~$\aaa$ reserves.
However, in that case necessarily $\gamma = \aaa
\ape 2$ and  $\gamma$  is initialized when $x$ is reserved, so
$x< s_{\gamma}$ contrary to our assumption.

Next, we verify that $ S_\gamma\cap (s_{\gamma}, \infty) \sub F_{e+1} $.
Suppose that $x \in S_\gamma$ and $x>s_{\gamma}$. Then $x \in S_\aaa$, so by
the inductive hypothesis $x \in F_e$. At a stage $s \ge s_\gamma$, an
element~$x$ of~$S_\aaa$ cannot be removed from $F_{e+1}$ by a strategy $\beta
>_L \aaa$ because $S_\beta\cap S_\aaa= \ES$ by~(3) as verified above and
since~$\beta$ can only remove elements from~$S_\beta$. So~$x$ can only be
removed from $F_{e+1}$ by~$\aaa \colon M_e$ or~$\aaa \colon P_e$.

If $\aaa \colon M_e$ removes~$x$ from~$F_{e+1}$, then $x \not \in \wt
S_\aaa$, which contradicts that $x \in S_\gamma$. So, by construction, the
only way~$x$ can be removed from $F_{e+1}$ is by the strategy $\aaa \colon
P_e$. Since~$x>s_{\gamma}$ this would mean that~$x$ does not enter~$S_\gamma$
either, contrary to our assumption.

\begin{claim}
Each requirement~$ M_e$ is met, namely, $F_e \setminus F_{e+1}$ is infinite.
\end{claim}

To see this, recall that $\aaa = g \uhr e$. The
action of $\aaa:M_e$ removes infinitely many elements of $S_\aaa $
from~$F_{e+1}$.
 This suffices  by Claim~\ref{cl:3}.
\begin{claim}
Each requirement~$ P_e$ is met.
\end{claim}

Suppose the hypothesis of~$P_e$ holds. Then every number that~$\aaa$ reserves
is eventually processed. So either $g(e)= 0$, in which case $F_{e+1} \sub^*
V_{e,0}$ by Claim~\ref{cl:3}, or $g(e)= 1$, in which case $F_{e+1} \sub^*
V_{e,1}$, also by Claim~\ref{cl:3}.
\end{proof}

\section{Ultrafilter bases for other Boolean algebras}
\label{sec:other_BAs}


As mentioned, we have set up our framework to apply to general countable
Boolean algebras, rather than merely the Boolean algebra of the computable
sets, mainly with subsequent research in mind. In this last section of our
paper, we provide two results in the setting of other Boolean algebras of
sets.


Recall that~$K(x)$ denotes the prefix-free complexity of a string~$x$, and
that a set $A\sub \NN$ is \emph{$K$-trivial} if $\ex c\, \forall n \, K(A\uhr
n) \le K(0^n)+c$. For more background on $K$-trivial sets, see Nies
\cite[Ch.\ 5]{Nies:book} or Downey and Hirschfeldt~\cite[Ch.\
11]{Downey.Hirschfeldt:book}. Note that by combining results of various
authors, the $K$-trivial degrees form a Turing ideal in the $\DII$-degrees
(see, e.g., Nies \cite[Sections 5.2, 5.4]{Nies:book}). Thus the $K$-trivial
sets form a Boolean algebra.

\begin{thm}
There is a $\DII$ ultrafilter base for the Boolean algebra of the $K$-trivial
sets.
\end{thm}

\begin{proof}
Ku\v cera and Slaman~\cite{Kucera.Slaman:09} noted that there is a function
$h \leT \Halt$ that dominates all functions that are partial computable in
some $K$-trivial set. We use~$h$ in a variation of the proof of
Theorem~\ref{thm: High}.

Let $\seq{V_{e,0}, V_{e,1}}\sN e$ be a uniform listing of the $K$-trivials
and their complements given by wtt-reductions to~$\Halt$; such a listing
exists by Downey, Hirschfeldt, Nies, and
Stephan~\cite{Downey.Hirschfeldt.ea:03} (see also \cite[Theorem
5.3.28]{Nies:book}).

Let $T=\{0,1\}^{< \omega}$. For each $\alpha \in T$, we define a (possibly
finite) $K$-trivial set~$S_\aaa$. Let $S_{\ES}= \NN$. Suppose we have defined
the set $S_\aaa = \{r_0< r_1 < \cdots \}$. Let $\wt S_\aaa$ contain the
numbers of the form $r_{2i}$. Let $S_{\aaa\ape k} = \wt S_\aaa \cap V_{e,k}$
for $e = |\alpha|$ and $k= 0,1$. Since $\wt S_\aaa \leT S_\aaa$, one verifies
inductively that all these sets are $K$-trivial.

Uniformly recursively in~$\Halt$, we build sets~$F_e$, given as the set of
members of nondecreasing unbounded sequences $a^e_0 \le a^e_1\le \ldots$.
Suppose we have defined~$a^e_{k-1}$. Try to let $\aaa \in T$ be the leftmost
string of length~$e$ such that~$S_\aaa$ has at least $k+1$ elements less than
$h(k)$. If such~$\aaa$ exists, let~$a^e_k$ be the~$k$-th element of~$S_\aaa$,
unless this is less than $a^e_{k-1}$, in which case we let $a^e_k=
a^e_{k-1}$.

%

Let $g \in \tp \omega$ denote the leftmost path in $\{0,1\}^\omega$ such that
for every~$e$, the set $S_{g \uhr e}$ is infinite. Fix~$e$ and let $\aaa =
g\uhr e$. Let $p(k)$ be the $(k+1)$-st element of~$S_{\aaa}$. Since~$h$
dominates the function~$p$, eventually in the definition of~$F_e$ we will
always pick~$\aaa$. Hence $F_e=^* S_\aaa$. In particular,~$F_e$ is
$K$-trivial. To see that $F\leT \Halt$, given input $n =\la r, e\ra$,
with~$\Halt$ as an oracle, compute the least~$k$ such that $r\le a^e_k$,
using that the sequences $\seq {a^e_k}\sN k$ are unbounded for each~$e$. Then
$n \in F$ iff $r= a^e_k$.

Clearly, if~$S_{\aaa}$ is infinite, then $S_\aaa \setminus S_\beta$ is
infinite for $\aaa \prec \beta$. So $ F_e \setminus F_{e+1}$ is infinite.

To verify that~$F$ is an ultrafilter base for the $K$-trivials, let~$R$ be a
$K$-trivial set. Pick~$e$ such that $R= V_{e, 0}$ and $\ol R= V_{e,1}$. If
$g(e)=0$ then $S_{g\, \uhr {e+1}} \sub V_{e,0}$, and hence $F_{e+1} \sub^*
R$. Otherwise, $S_{g\, \uhr {e+1}} \sub V_{e,1}$, and hence $F_{e+1} \sub^*
\ol R$.
\end{proof}

\begin{remark}
Any ultrafilter base for the $K$-trivials must have high degree. We can see
this by modifying the proof of Theorem~\ref{thm:UFB->high}: Every
Martin-L\"of random set~$X$ is Martin-L\"of random relative to every
$K$-trivial (i.e., $K$-trivial sets are \emph{low for ML-randomness}). Hence
neither~$X$ nor~$\ol X$ contains an infinite $K$-trivial subset.
\end{remark}

\smallskip

Finally, we consider the Boolean algebra of the primitive recursive sets. One
says that an oracle~$L$ is of PA degree if it computes a completion of Peano
arithmetic. Recall that~$L$ is of PA degree if and only if it computes a
separating set for each disjoint pair of c.e.\ sets.

\begin{thm}
An oracle~$C$ computes an ultrafilter base for the primitive recursive sets
if and only if~$C'$ is of PA degree relative to~$\Halt$.
\end{thm}

\begin{proof}
We modify the proof of Jockusch and
Stephan~\cite[Theorem~2.1]{Jockusch.Stephan:93}. They say that a set $S \sub
\NN$ is \emph{$p$-cohesive} if~$S$ is cohesive for the primitive recursive
sets. Their theorem states that~$S$ is $p$-cohesive if and only if~$S'$ is of
PA degree relative to~$\Halt$.

\rapf Suppose that~$C$ computes an ultrafilter base~$F$ for the primitive
recursive sets. Let $g\leT F$ be a function associated with~$F$ as in
Definition~\ref{df:assoc_fcn}. Then the range~$S$ of~$g$ is $p$-cohesive.
Hence~$S'$ and therefore~$C'$ is of PA degree relative to~$\Halt$ by one
implication of~\cite[Theorem~2.1]{Jockusch.Stephan:93}.

\lapf We modify the proof of the other implication
of~\cite[Theorem~2.1]{Jockusch.Stephan:93}. Let $\seq {A_i}\sN i$ be a
uniformly recursive list of all the primitive recursive sets. We call~$i$ a
\emph{primitive recursive index for~$A_i$} (or \emph{index}, for short). By
our hypothesis on~$C$, there is a function $g\leT C'$ such that
\begin{eqnarray*}
|A_i \cap A_n| < |A_i \cap \ol A_n| & \RA & g(i,n) = 0 \\
|A_i \cap \ol A_n| < |A_i \cap A_n| & \RA & g(i,n) = 1
\end{eqnarray*}
(because the conditions on the left are both~$\SI 2$, and so~$C'$ computes a
separating set for them).

We inductively define a $C'$-computable sequence of indices $\seq{e_n} \sN
n$. Let~$e_0$ be an index for~$\NN$. If~$e_n$ has been defined and $A_{e_n} =
\{r_0 < r_1< \cdots\}$ (possibly finite), let~$e'_n$ be an index, uniformly
obtained from~$e_n$, such that $A_{e'_n}= \{r_0, r_2, \ldots\}$. Now let
\medskip

$A_{e_{n+1}}= A_{e'_n} \cap \ol A_n$ if $g(e'_n, n )= 0$, and

$A_{e_{n+1}}= A_{e'_n} \cap A_n$ if $g(e'_n, n )= 1$.

\medskip

By induction on~$n$, one verifies that $A_{e_n}$ is infinite and $A_{e_n}
\setminus A_{e_{n+1}}$ is infinite. Since $g\leT C'$, the numbers~$e_n$ have
a uniformly $C$-computable approximation $\seq{e_{n,x}} \sN x$.

Let the ultrafilter base $F\leT C$ be given by $F_n(x)= A_{e_{n,x}}(x)$. Then
$F_n =^* A_{e_n}$ is primitive recursive. Since $F_{n+1} \sub^* \ol A_n$ or
$F_{n+1} \sub^* A_n$ for each~$n$, the set~$F$ is an ultrafilter base for the
primitive recursive sets.
\end{proof}

\def\cprime{$'$}

\end{document}